\documentclass[11pt]{amsart}

\usepackage{amsmath}
\usepackage{fullpage}
\usepackage{xspace}
\usepackage[psamsfonts]{amssymb}
\usepackage[latin1]{inputenc}
\usepackage{graphicx,color}
\usepackage[curve]{xypic} 
\usepackage{hyperref}
\usepackage{graphicx}

\usepackage{amsmath}%
\usepackage{amsthm}%
\usepackage{amscd}
\usepackage{amsfonts}%
\usepackage{amssymb}%
\usepackage{graphicx}

\usepackage{mathrsfs}

\usepackage{tikz}
\usetikzlibrary{matrix,arrows}

\newtheorem{theorem}{Theorem}[section]

\newtheorem{conjecture}[theorem]{Conjecture}
\newtheorem{corollary}[theorem]{Corollary}

\newtheorem{lemma}[theorem]{Lemma}

\newtheorem{proposition}[theorem]{Proposition}

\theoremstyle{remark}

\numberwithin{equation}{section}

\newcommand{\pfrak}{\mathfrak{p}}

\newcommand{\nfrak}{\mathfrak{n}}

\newcommand{\Pcal}{\mathscr{P}}
\newcommand{\Qcal}{\mathscr{Q}}

\newcommand{\Scal}{\mathscr{S}}

\newcommand{\Z}{\mathbb{Z}}
\newcommand{\C}{\mathbb{C}}

\newcommand{\F}{\mathbb{F}}
\newcommand{\Q}{\mathbb{Q}}

\newcommand{\Aut}{\mathrm{Aut}}

\newcommand{\rk}{\mathrm{rank}\,}

\newcommand{\Hom}{\mathrm{Hom}}
\newcommand{\Gal}{\mathrm{Gal}}
\newcommand{\Frob}{\mathrm{Frob}}
\newcommand{\tr}{\mathrm{tr}}

\newcommand{\Tam}{\mathrm{Tam}}
\newcommand{\Sel}{\mathrm{Sel}}

  \DeclareFontFamily{U}{wncy}{}
    \DeclareFontShape{U}{wncy}{m}{n}{<->wncyr10}{}
    \DeclareSymbolFont{mcy}{U}{wncy}{m}{n}
    \DeclareMathSymbol{\Sha}{\mathord}{mcy}{"58}

\begin{document}
\title[]{Towards Hilbert's Tenth Problem for rings of integers through Iwasawa theory and Heegner points}

\author{Natalia Garcia-Fritz}
\address{ Departamento de Matem\'aticas\newline
\indent Pontificia Universidad Cat\'olica de Chile\newline
\indent Facultad de Matem\'aticas\newline
\indent 4860 Av. Vicu\~na Mackenna\newline
\indent Macul, RM, Chile}
\email[N. Garcia-Fritz]{natalia.garcia@mat.uc.cl}

\author{Hector Pasten}
\address{ Departamento de Matem\'aticas\newline
\indent Pontificia Universidad Cat\'olica de Chile\newline
\indent Facultad de Matem\'aticas\newline
\indent 4860 Av. Vicu\~na Mackenna\newline
\indent Macul, RM, Chile}
\email[H. Pasten]{hector.pasten@mat.uc.cl}%

\thanks{N. G.-F. was supported by the FONDECYT Iniciaci\'on en Investigaci\'on grant 11170192, and the CONICYT PAI grant 79170039. H.P. was supported by FONDECYT Regular grant 1190442.}
\date{\today}

\begin{abstract} For a positive proportion of primes $p$ and $q$, we prove that $\Z$ is Diophantine in the ring of integers of $\Q(\sqrt[3]{p},\sqrt{-q})$. This provides a new and explicit infinite family of number fields $K$ such that Hilbert's tenth problem for $O_K$ is unsolvable. Our methods use Iwasawa theory and congruences of Heegner points in order to obtain suitable rank stability properties for elliptic curves.
\end{abstract}

\maketitle



\section{Introduction} 

A celebrated result due to Matiyasevich \cite{Matiyasevich} after the work of Davis, Putnam, and Robinson \cite{DPR}, shows that computably enumerable sets over $\Z$ are exactly the same as Diophantine sets over $\Z$. As a consequence, Hilbert's tenth problem is unsolvable: namely, there is no algorithm (Turing machine) that takes as input polynomial equations over $\Z$ and decides whether they have integer solutions. 

A natural extension is the analogue of Hilbert's tenth problem for rings of integers of number fields, which remains as one of the main open problems in the area. In the direction of a negative solution, the following conjecture formulated by Denef and Lipshitz \cite{DenLip} is widely believed (the notion of Diophantine set is recalled in Section \ref{SecDioph}):
\begin{conjecture} \label{MainConj} Let $K$ be a number field $K$. Then $\Z$ is a Diophantine subset of $O_K$.
\end{conjecture}
By a standard argument, if $\Z$ is Diophantine in $O_K$ for a number field $K$, then the analogue of Hilbert's tenth problem for $O_K$ has a negative solution, and moreover, the Diophantine sets over $O_K$ are the same as computably enumerable subsets. Thus, the efforts have concentrated in proving new cases of Conjecture \ref{MainConj}. 

However, progress on Conjecture \ref{MainConj} has been slow and extremely difficult. First, one has the following known cases dating back to the eighties:
\begin{itemize}
\item[(i)] $K$ is totally real or a quadratic extension of a totally real field \cite{DenLip, Denef}
\item[(ii)] $[K:\Q]=4$, $K$ is not totally real, and $K/\Q$ has a proper intermediate field \cite{DenLip}
\item[(iii)] $K$ has exactly one complex place \cite{Pheidas, ShlH10, Videla}
\item[(iv)] $K$ is contained in a number field $L$ such that $\Z$ is Diophantine in $O_L$; in particular, this holds when $K/\Q$ is abelian \cite{ShaShl}
\end{itemize}
 Later work by Poonen \cite{Poonen}, Cornelissen-Pheidas-Zahidi \cite{CPZ}, and Shlapentokh \cite{ShlE} reduced Conjecture \ref{MainConj} to the existence of elliptic curves preserving its (positive) rank in suitable extensions of number fields (cf. Section \ref{SecDioph}). This allowed to verify Conjecture \ref{MainConj} on a case-by-case basis for concrete examples of number fields by exhibiting a suitable elliptic curve (see the Paragraph \ref{SubSecIntDioph} for the worked-out example $K=\Q(\sqrt[5]{2})$). Besides specific examples, the elliptic curve criteria also led  to the following more recent result by Mazur and Rubin:
\begin{itemize} 
\item[(v)] For any number field $F$ for which $\Z$ is Diophantine in $O_F$, there is a positive proportion of primes $\ell$ such that for all $n\ge 1$, there are infinitely many choices of $K$ as a cyclic $\ell^n$-extension of $F$ for which $\Z$ is Diophantine in $O_K$. (cf.\ \cite{MazRubDS}; see also \cite{MazRubH10}).
\end{itemize}
This concludes our summary of known cases of Conjecture \ref{MainConj}.


In this work we develop a method to attack Conjecture \ref{MainConj} which allows us to prove it in several new cases. Our main result is the following:
\begin{theorem} \label{ThmMainIntro} There are explicit Chebotarev sets of primes $\Pcal$ and  $\Qcal$ having positive density in the primes, such that for all $p\in \Pcal$ and all $q\in \Qcal$ we have that $\Z$ is Diophantine in the ring of integers of $\Q(\sqrt[3]{p}, \sqrt{-q})$.  In particular, for these fields $K=\Q(\sqrt[3]{p}, \sqrt{-q})$, the analogue of Hilbert's tenth problem for $O_K$ is unsolvable.

The sets $\Pcal$ and $\Qcal$ can be chosen to have densities $5/16>31\%$ and $1/12>8\%$ respectively.

\end{theorem}
The sets of prime numbers $\Pcal$ and $\Qcal$ from our construction will be made explicit in Section \ref{SecProof}.

Theorem \ref{ThmMainIntro} precisely covers one of the simplest kinds of number fields which is out of the scope of all the available results on Hilbert's tenth problem for rings of integers ---see Paragraph \ref{SubsecComparison} for details. 

Besides the fact that we prove new cases of Conjecture \ref{MainConj}, the main novelty in our work is the method of proof for Theorem \ref{ThmMainIntro}. Using Shlapentokh's elliptic curve criterion and the result of Pheidas, Shlapentokh and Videla quoted above in item (iii), we reduce the problem to the existence of a suitable elliptic curve $E$ over $\Q$ with $\rk E(\Q(\sqrt[3]{p}))=0$ and $\rk E(\Q(\sqrt{-q}))>0$. The first condition is achieved by a study of the variation of cyclotomic Iwasawa invariants under base-change, while the second one is achieved by means of a recent result by Kriz and Li \cite{KrizLi} on congruences of Heegner points in the context of the celebrated Goldfeld's conjecture. Both methods ---Iwasawa theory and Heegner points--- impose a series of technical conditions on the admissible elliptic curves $E$, and a critical aspect of the proof of Theorem \ref{ThmMainIntro} is to make sure that there exist elliptic curves that are admissible for both methods simultaneously.

Let us remark that the approach by Mazur and Rubin (cf.\ Item (v) above) is substantially different to ours: they achieve the necessary rank-stability conditions by delicate cohomological computations to modify Selmer ranks under quadratic twists.

We conclude this introduction by recalling that there is strong evidence for Conjecture \ref{MainConj}, at least if one assumes some standard conjectures on elliptic curves. On the one hand, Mazur and Rubin \cite{MazRubH10} proved that it follows from the squareness conjecture for the $2$-torsion of Shafarevich-Tate groups of elliptic curves over number fields. On the other hand,  Murty and Pasten \cite{MurtyPasten} proved that it follows from rank aspects of the Birch and Swinnerton-Dyer conjecture.


\section{Preliminaries on Iwasawa theory}

Our main results are relevant on logic aspects of number theory but part of our arguments require some Iwasawa theory. So, it might be useful to include here a brief reminder of the latter subject. Other than for the purpose of checking the notation, experts can safely skip this section. 

In this section $\ell$ denotes a prime number. We specialize to the case of cyclotomic Iwasawa theory for elliptic curves, which suffices for our purposes. All results discussed in this section can be found in \cite{MazurIwa} and \cite{Greenberg}.

\subsection{Algebra} \label{ParAlg} It is a standard fact that if $X$ is a finitely generated torsion $\Z_{\ell}[[T]]$-module, then $X$ is pseudo-isomorphic\footnote{A pseudo isomorphism is a morphism with finite kernel and cokernel.} to 
$$
\bigoplus_{i}\Z_{\ell}[[T]]/(\ell^{\mu_i})\oplus \bigoplus_j \Z_{\ell}[[T]]/(f_j^{m_j})
$$
where $\mu_i$ are positive integers and $f_j\in\Z_{\ell}[T]$ are monic, irreducible polynomials such that all non-leading coefficients are in $\ell\Z_{\ell}$.

The above decomposition is unique up to order, and one defines the following invariants of $X$:
\begin{itemize}
\item $\mu_X=\sum_i \mu_i$ the $\mu$-invariant
\item $f_X= \ell^{\mu_X}\prod_j f_j^{m_j} $ the characteristic polynomial
\item $\lambda_X=\deg f_X$ the $\lambda$-invariant.
\end{itemize}


\subsection{Iwasawa algebras and modules}\label{ParGal} For a number field $k$, we let $k_\infty$ be the maximal cyclotomic $\Z_\ell$-extension of $k$. For each $m\ge 0$ we let  $k_m/k$ for the unique cyclotomic $\Z/\ell^m\Z$-extension of $k$ contained in $k_\infty$, and we observe that $k_0=k$ and $k_\infty=\cup_m k_m$. The Galois group $\Gamma_k$ of $k_\infty/k$ is (non-canonically) isomorphic to $\Z_{\ell}$ as a profinite group. The \emph{Iwasawa algebra}  $\Lambda_k=\Z_{\ell}[[\Gamma_k]]$ is the profinite completion of the group algebra $\Z_{\ell}[\Gamma_k]$. The choice of a topological generator $\gamma\in\Gamma_k$ determines a continuous $\Z_{\ell}$-algebra isomorphism $\Psi_{k,\gamma}:\Lambda_k\to \Z_{\ell}[[T]]$ by the rule $\gamma\mapsto 1+T$.

Thus, one can study finitely generated torsion $\Lambda_k$-modules by means of the classification theorem of the previous paragraph. Using Selmer groups of elliptic curves, one can construct certain finitely generated modules $X(E/k_\infty)$ that are believed to be $\Lambda_k$-torsion.


\subsection{Selmer groups} Let $k$ be a number field and let $E$ be an elliptic curve over $k$. For any algebraic extension $L/k$ (not necessarily finite) the $\ell$-primary part of the Selmer group of $E$ over $L$ is denoted by $\Sel_{\ell^\infty}(E/L)$. It fits in the exact sequence
\begin{equation}\label{EqnSha}
0\to E(L)\otimes_\Z\Q_{\ell}/\Z_{\ell}\to \Sel_{\ell^\infty}(E/L)\to \Sha(E/L)[\ell^\infty]\to 0
\end{equation}
where $\Sha(E/L)$ is the Shafarevich-Tate group. 

It is known that $\Sel_{\ell^\infty}(E/k)\simeq (\Q_\ell/\Z_\ell)^{\rho_\ell(E/k)}\oplus G$ for certain integer $\rho_\ell(E/k)\ge 0$, where $G$ is a finite group. The number $\rho_\ell(E/k)$ is the co-rank of $\Sel_{\ell^\infty}(E/k)$ and the previous exact sequence gives $\rk E(k)\le \rho_\ell(E/k)$. Of course, if $\Sha(E/k)[\ell^\infty]$ is finite (as conjectured) then $\rk E(k)= \rho_\ell(E/k)$.

A foundational result of Mazur allows one to recover valuable arithmetic information about $E$ over all the number fields $k_m$ from the $\Gamma_k$-action on $\Sel_{\ell^\infty}(E/k_\infty)$.
\begin{theorem}[Mazur's control theorem] Let $E$ be an elliptic curve over $k$ having good ordinary reduction at each prime of $k$ above $\ell$. The natural maps
$$
\Sel_{\ell^\infty}(E/k_m)\to \Sel_{\ell^\infty}(E/k_\infty)^{{\rm Gal}(k_\infty/k_m)}
$$
have finite kernels and cokernels, whose orders are bounded as $m$ varies. 
\end{theorem}


\subsection{Iwasawa modules attached to elliptic curves} The group $\Sel_{\ell^\infty}(E/k_\infty)$ has a natural $\Lambda_k$-module structure and it is topologically discrete. Giving $\Q_{\ell}/\Z_{\ell}$ the trivial $\Gamma_k$-action, we have an induced $\Lambda_k$-module structure on
$$
X(E/k_\infty):=\Hom_{\mathrm{cont}} (\Sel_{\ell^\infty}(E/k_\infty), \Q_{\ell}/\Z_{\ell}).
$$
By Pontryagin duality, $X(E/k_\infty)$ is compact. It is a standard result that $X(E/k_\infty)$ is finitely generated as a $\Lambda_k$-module.

\begin{conjecture}[Mazur] If $E$ has good ordinary reduction at each prime of $k$ above $\ell$, then $X(E/k_\infty)$ is a torsion $\Lambda_k$-module. 
\end{conjecture}

The general case of this conjecture remains open. For us it suffices to know the following result, which is a consequence of Mazur's control theorem.
\begin{theorem}[Mazur] \label{ThmMazurFin} If $E$ has good ordinary reduction at each prime of $k$ above $\ell$, and if $\Sel_{\ell^\infty}(E/k)$ is finite, then $X(E/k_\infty)$ is a torsion $\Lambda_k$-module. 
\end{theorem}

\subsection{Iwasawa invariants}
Suppose that $X(E/k_\infty)$ is a torsion $\Lambda_k$-module and fix a topological generator $\gamma\in \Gamma_k$. Then we can regard $X(E/k_\infty)$ as a finitely generated torsion $\Z_{\ell}[[T]]$-module via $\Psi_{k,\gamma}$ (cf.\ Paragraph \ref{ParGal}), and we have the associated invariants  $\mu_{E/k}:=\mu_{X(E/k_\infty)}$, $\lambda_{E/k}:=\lambda_{X(E/k_\infty)}$,  and $f_{E/k}:=f_{X(E/k_\infty)}$. The $\mu$ and $\lambda$ invariants are well-defined integers that only depend on $\ell$ and $E/k$. However, the characteristic polynomial $f_{E/k}\in\Z_\ell[T]$ also depends on the choice of topological generator $\gamma\in\Gamma_k$. 

Directly from the definitions one has
\begin{lemma}\label{Lemmalm} Suppose that $X(E/k_\infty)$ is a torsion $\Lambda_k$-module. Then:
\begin{itemize}
\item $\mu_{E/k}=0$ if and only if $X(E/k_\infty)$ is finitely generated as a $\Z_{\ell}$-module,
\item $\mu_{E/k}=\lambda_{E/k}=0$ if and only if $X(E/k_\infty)$ is a finite group.
\end{itemize}
\end{lemma}
It is quite common that the $\mu$-invariant vanishes, but this is not always the case. On the other hand, we will be mostly interested on the $\lambda$-invariant due to the following important consequence of Mazur's control theorem (cf.\ Theorem 1.9 in \cite{Greenberg}).

\begin{theorem} \label{ThmLambdaRk} Suppose that $E$ has good ordinary reduction at each prime of $k$ above $\ell$ and that $X(E/k_\infty)$ is a torsion $\Lambda_k$-module. Then for each $m\ge 0$ we have 
$$
\rk E(k_m)\le \rho_\ell(E/k_m)\le  \lambda_{E/k}.
$$
\end{theorem}


\subsection{The value of $f_{E/k}(0)$} We now discuss the characteristic polynomial $f_{E/k}$ assuming, of course, that $X(E/k_\infty)$ is $\Lambda_k$-torsion and that a topological generator $\gamma\in\Gamma_k$ is chosen.
\begin{proposition}\label{Propf0}
Suppose that $E$ has good ordinary reduction at each prime of $k$ above $\ell$ and that $X(E/k_\infty)$ is a torsion $\Lambda_k$-module. We have that $\Sel_{\ell^\infty}(E/k)$ is finite if and only if $f_{E/k}(0)\ne 0$.
\end{proposition}
\begin{proof} By Mazur's control theorem, $\Sel_{\ell^\infty}(E/k)$ is finite if and only if $\Sel_{\ell^\infty}(E/k_\infty)^{\Gamma_k}$ is finite. Choosing a topological generator $\gamma\in\Gamma_k$ and recalling that $\Psi_{k,\gamma}(\gamma)=1+T$, we see that the latter is equivalent to the finiteness of $X(E/k_\infty)/T\cdot X(E/k_\infty)$ by duality between $\Sel_{\ell^\infty}(E/k_\infty)$ and $X(E/k_\infty)$. Finally, we note that $X(E/k_\infty)/T\cdot X(E/k_\infty)$ is finite if and only if $T\nmid f_{E/k}$ (cf. Paragraph \ref{ParAlg}).
\end{proof}

It is of great interest to evaluate $f_{E/k}(0)\in \Z_{\ell}$ when it is non-zero. The precise value depends on the choice of topological generator $\gamma\in\Gamma_k$, but nevertheless the $\ell$-adic valuation of $f_{E/k}(0)$ turns out to depend only on $\ell$, the number field $k$, and the elliptic curve $E/k$. Before giving the formula, we need some notation.

Given a non-archimedean place $v$ of $k$, we write $\F_v$ for the residue field and $k_v$ for the completion of $k$ at $v$. If $E$ is an elliptic curve over $k$, the reduction of $E$ at $v$ is denoted by $\tilde{E}_v$ and its non-singular locus is $\tilde{E}_v^{ns}$. We write $E(k_v)_0$ for the subgroup of $E(k_v)$ consisting of local points whose reduction at $v$ is in $\tilde{E}_v^{ns}$. The Tamagawa factor at $v$ is defined by $c_v(E/k)=[E(k_v):E(k_v)_0]$, and the the product  $\prod_v c_v(E/k)$ over all non-archimedean places $v$ of $k$ is denoted by $\Tam(E/k)$. Let us write $a\sim_{\ell} b$ if $a,b\in \Q_{\ell}$ have the same $\ell$-adic valuation.

Note that the finiteness assumption on $\Sel_{\ell^\infty}(E/k)$  not only implies $f_{E/k}(0)\ne 0$ (cf.\ Theorem \ref{ThmMazurFin} and Proposition \ref{Propf0}), but also, it implies that both $\Sha(E/k)[\ell^{\infty}]$ and $E(k)$ are finite (cf. the exact sequence \eqref{EqnSha}). So the formula in the next result only involves non-zero finite numbers. See also \cite{PR} and \cite{Schneider}.
\begin{theorem}[cf.\ Theorem 4.1 in \cite{Greenberg}] \label{Thmf0} Suppose that $E$ has good ordinary reduction at each prime of $k$ above $\ell$ and that $\Sel_{\ell^\infty}(E/k)$ is finite. Then $f_{E/k}(0)\in\Z_\ell$ is non-zero and it satisfies
$$
f_{E/k}(0)\sim_{\ell}   \frac{\Tam(E/k)\cdot  \#\Sha(E/k)[\ell^{\infty}] }{\#E(k)^2} \cdot \prod_{v|\ell} \#\tilde{E}_v(\F_v)^2. 
$$
\end{theorem}
The proof of Theorem \ref{Thmf0} relies on the theory of the cyclotomic Euler characteristic, which we don't review here.

 

\section{A general construction of integrally Diophantine extensions}\label{SecDioph}

\subsection{Integrally Diophantine extensions} \label{SubSecIntDioph} Given a  commutative unitary ring $A$ and a positive integer $n$, let us recall that a subset $S\subseteq A^n$ is \emph{Diophantine} over $A$ if for some $m\ge 0$ there are polynomials $F_1,\ldots, F_k \in A[x_1,...,x_n,y_1,...,y_m]$ such that
$$
S=\{{\bf a}\in A^n : \exists {\bf b}\in A^m \mbox{ such that for each }1\le j\le k \mbox{ we have }  F_j({\bf a},{\bf b})=0\}.
$$

An extension of number fields $K/F$ is \emph{integrally Diophantine} if $O_F$ is Diophantine in $O_K$.   It is a standard fact that if $K/F/L$ is a tower of number fields and both $F/L$ and $K/F$ are integrally Diophantine, then so is $K/L$. In particular, 
\begin{lemma}\label{LemmaTransfer} Let $K/F$ be an extension of number fields. If $K/F$ is integrally Diophantine and $\Z$ is Diophantine in $O_F$, then $\Z$ is Diophantine in  $O_K$.
\end{lemma}
As discussed in the introduction, there are elliptic curve criteria developed by Poonen, Cornelissen-Pheidas-Zahidi, and Shlapentokh to prove that a given extension of number fields is integrally Diophantine. For our purposes, let us recall here Shlapentokh's criterion.

\begin{theorem}[Shlapentokh \cite{ShlE}] \label{ThmShl} Let $K/F$ be an extension of number fields. Suppose that there is an elliptic curve $E$ defined over $F$ such that $\rk E(K)=\rk E(F)>0$. Then $K/F$ is integrally Diophantine.
\end{theorem}

This result can be applied on a case-by-case basis to concrete examples of number fields, but it is not known at present how to systematically apply it in general ---at least, not without assuming some conjecture on elliptic curves (cf. \cite{MazRubH10} and \cite{MurtyPasten}).
For instance, the number field $\Q(\sqrt[5]{2})$ is not covered by the results (i)-(v) quoted in the introduction. However, the elliptic curve $E$ of affine equation $y^2 + xy = x^3 - x^2 - x + 1$ (Cremona label $58a1$) satisfies
$$
\rk E(\Q(\sqrt[5]{2}))=\rk E(\Q)=1,
$$
as it can be readily checked on Sage. By Shlapentokh's theorem we get that $\Q(\sqrt[5]{2})/\Q$ is integrally Diophantine (in fact, already the results of Poonen \cite{Poonen} or Cornelissen-Pheidas-Zahidi \cite{CPZ} suffice here).  That is, Conjecture \ref{MainConj} holds for $K=\Q(\sqrt[5]{2})$.


\subsection{A general construction} The general case of Conjecture \ref{MainConj} ---and in fact, of Hilbert's tenth problem for rings of integers of number fields--- remains open, as discussed in the introduction. It is therefore of great importance to develop tools that permit to construct integrally Diophantine extensions. The next result, despite its simple proof, gives a general way to construct such extensions and to prove new cases of Conjecture \ref{MainConj}. 

\begin{proposition} \label{PropMain} Let $F/M$ and $L/M$ be extensions of number fields with $L/M$ quadratic. Let $K=F.L$ be the compositum of $F$ and $L$ over $M$.  Suppose that there is an elliptic curve $E$ over $M$ satisfying the conditions: 
\begin{itemize}
\item[(i)] $\rk E(F)=0$
\item[(ii)] $\rk E(L)>0$.
\end{itemize}
Then $K/F$ is integrally Diophantine.
\end{proposition}
\begin{proof}
Let $E^L$ be the elliptic curve over $M$ defined as the quadratic twist of $E$ by $L$.  By (i) and (ii) we see that $L$ is not contained in $F$, so the extension $K/F$ is quadratic and
$$
\rk E^L(K)=\rk E(F)+ \rk E^L(F)=\rk E^L(F).
$$ 
Since $\rk E^L(K) =\rk E(K)\ge \rk E(L)>0$ we can apply Theorem \ref{ThmShl} to the extension $K/F$ with the elliptic curve $E^L$ over $F$.
\end{proof}

As an example of how this construction readily leads to new cases of Conjecture \ref{MainConj}, let us point out the following simple consequence.
\begin{proposition}\label{PropEasy} Let $p$ be a prime of the form $p\equiv 3\bmod 4$. Then $\Z$ is Diophantine in the ring of integers of $\Q(\sqrt[5]{2}, \sqrt{p})$.
\end{proposition}
\begin{proof} We apply Proposition \ref{PropMain} with $M=\Q$, $F=\Q(\sqrt[5]{2})$, and $L=\Q(\sqrt{p})$ for $p\equiv 3\bmod 4$ a prime number. Consider the elliptic curve $E$ over $\Q$ with affine equation $y^2=x^3-4x$; this is the twist by $2$ of the celebrated Congruent Number elliptic curve $y^2=x^3-x$. By classical results due to Heegner \cite{Heegner} and Birch \cite{Birch} (see also \cite{Monsky}) we have that $2p$ is a congruent number when $p\equiv 3\bmod 4$. Hence $\rk E(\Q(\sqrt{p}))>0$ for these primes $p$. On the other hand, a direct computation on Sage shows that $\rk E(\Q(\sqrt[5]{2}))=0$. Therefore $K=\Q(\sqrt[5]{2}, \sqrt{p})$ is integrally Diophantine over $F=\Q(\sqrt[5]{2})$. In Paragraph \ref{SubSecIntDioph} we already checked that $\Q(\sqrt[5]{2})/\Q$ is integrally Diophantine, hence $K/\Q$ also is (cf.\ Lemma \ref{LemmaTransfer}). \end{proof}

In Proposition \ref{PropEasy}, all the number fields that we obtain are quadratic extensions of a single number field ---namely, of $\Q(\sqrt[5]{2})$--- and of course one can obtain many other results of this sort thanks to Proposition \ref{PropMain}. Theorem \ref{ThmMainIntro} instead is much more delicate and the proof lies deeper, as it concerns number fields of the form $\Q(\sqrt[3]{p},\sqrt{-q})$ where the two parameters $p$ and $q$ can be chosen independently of each other.


\subsection{Comparison with other available results}\label{SubsecComparison}
 
The following elementary lemma will help us to check that the families of number fields $K$ for which we prove that $\Z$ is Diophantine in $O_K$ are indeed new.
\begin{lemma} \label{LemmaNew} Let $\ell>2$ be a prime and let $F/\Q$ be a degree $\ell$ extension which is not totally real.  Then $F$ is not contained in a quadratic extension of a totally real number field.
\end{lemma}
\begin{proof}
Suppose that $H/N$ is a quadratic extension with $N$ a totally real number field, such that $F\subseteq H$. Observe that $N$ and $F$ are linearly disjoint over $\Q$, for otherwise an element $\gamma\in F-\Q$ would be conjugate to an element of $N$, but $F=\Q(\gamma)$ is not totally real. It follows that the compositum $F.N$ has degree $\ell$ over $N$, but this is not possible since $F.N\subseteq H$ and $[H:N]=2$.
\end{proof}
For instance, let us verify that the number fields in Proposition \ref{PropEasy} are not included in number fields  covered by the results in items (i), (ii), (iii) from the introduction: For a prime $p$, the number field $\Q(\sqrt[5]{2}, \sqrt{p})$ has degree $10$ over $\Q$ and it has $4$ complex places, hence $K$ is not contained in fields for which (ii) or (iii) apply. By Lemma \ref{LemmaNew}, $K$ is not contained in a quadratic extension of a totally real field, so (i) does not apply.

Regarding item (v), we note that the quadratic extensions $\Q(\sqrt[5]{2}, \sqrt{p})/\Q(\sqrt[5]{2})$ are obtained by adjoining the square root of a rational prime, as opposed to using unrestricted primes from $\Q(\sqrt[5]{2})$. So the method of proof in \cite{MazRubH10, MazRubDS} (cf.\ item (v) above) does not apply either, as the auxiliary primes required by the methods of \emph{loc.\ cit.} cannot be guaranteed to be chosen in $\Q$.

Finally, we remark that Lemma \ref{LemmaNew} allows us to check in a similar way that the number fields considered in our main result Theorem \ref{ThmMainIntro} are out of the scope of the available results in the literature. In fact, if $p$ and $q$ are prime numbers, the number field $\mathbb{Q}(\sqrt[3]{p},\sqrt{-q})$ has degree $6$ over $\mathbb{Q}$, it has three complex places, and it contains the  non-totally real number field $\mathbb{Q}(\sqrt[3]{p})$.


\section{Preserving rank zero: Iwasawa theory}

For a positive integer $n$, the set of complex $n$-th roots of $1$ is denoted by $\mu_n$. The purpose of this section is to prove the following result
\begin{theorem}\label{ThmRkZero} Let $\ell>2$ be a prime. Let $E$ be an elliptic curve over $\Q$ of conductor $N$ and assume the following:
\begin{enumerate}
\item \label{A} $E$ has good ordinary reduction at $\ell$
\item \label{B} the residual Galois representation $\rho_{E[\ell]}:G_\Q\to \Aut (E[\ell])$ is surjective
\item \label{C} $\rk E(\Q(\mu_\ell))=0$
\item \label{D} $\Sha(E/\Q(\mu_\ell))[\ell]=(0)$
\item \label{E} $\ell\nmid \Tam(E/\Q(\mu_\ell))\cdot  \#\tilde{E}_\ell(\F_\ell)$.
\end{enumerate}
Consider the set of prime numbers
$$
\Pcal(E,\ell)=\{p : p\nmid N\mbox{, }p \equiv 1\bmod \ell\mbox{, and }a_p(E)\not\equiv 2\bmod \ell\}.
$$
Then $\Pcal(E,\ell)$ is a Chebotarev set of primes of density 
$$
\delta(\Pcal(E,\ell))=\frac{\ell^2-\ell-1}{(\ell-1)(\ell^2-1)}>0
$$
and for every $\ell$-power free integer $a>1$ supported on $\Pcal(E,\ell)$ we have that $\Sel_{\ell^\infty}(E/\Q(\mu_\ell,\sqrt[\ell]{a}))$ is finite. In particular, for these integers $a$ we have $\rk E(\Q(\sqrt[\ell]{a}))=\rk E(\Q(\mu_\ell,\sqrt[\ell]{a}))=0$.
\end{theorem}
For our applications on Hilbert's tenth problem we (crucially) need the case $\ell=3$. Nevertheless, this more general theorem can be of independent interest. We also remark that the various hypotheses in the statement are amenable for computations; a concrete example is presented in Section \ref{SecProof}. Condition \eqref{B} is convenient for showing that $\Pcal(E,\ell)$ has positive density, but it can certainly be relaxed ---however, it is enough for our purposes.

We remark that a related result is sketched as Theorem 18 in \cite{VladD} for primes $\ell\ge 5$ without verifying the existence of infinitely many integers $a$. For $\ell=3$, the special case of $E=X_1(11)$ is worked out in \cite{TimD} where the existence of infinitely many integers $a$ is shown by explicit computations with an equation for the modular curve $X_1(11)$.

Chao Li has pointed out to us another possible approach to achieve rank $0$ over cubic extensions, at least for a particular kind of elliptic curves. Namely, given a Mordell elliptic curve $E: y^2=x^3+k$ (i.e. an elliptic curve of $j$-invariant $0$) defined over $\Q$ with $\rk E(\Q)=0$, we have $\rk E(\Q(\sqrt[3]{p}))=0$ if $\rk E^{[p]}(\Q)=0$ and $\rk E^{[p^2]}(\Q)=0$, where $E^{[d]}:y^2=x^3+d^2k$ is the cubic twist by $d$. The theory of \cite{KrizLi} allows one to achieve $\rk E^{[p]}(\Q)=0$ for many primes $p$ under suitable conditions. For our applications it would be of great interest to extend this theory in order to control the two relevant cubic twists simultaneously. We point out that both cubic twists $E^{[p]}$ and $E^{[p^2]}$ are needed; for instance, for $E: y^2=x^3+1$ and $p=5$ we have $\rk E(\Q)=0$, $\rk E^{[5]}(\Q)=0$ but $\rk E(\Q(\sqrt[3]{5}))=1$.


\subsection{Variation of the $\lambda$-invariant} In view of Theorem \ref{ThmLambdaRk}, we will be interested in the variation of the $\lambda$-invariant of an elliptic curve under base change. An important result by Hachimori and Matsuno \cite{HaMa} gives a precise formula under suitable assumptions. Here we state a special case. As usual, for a number field $k$ the set of places of $k$ is denoted by $M_k$.

\begin{theorem}[Hachimori-Matsuno] \label{ThmHM} Let $\ell>2$ be a prime and let $k$ be a number field. Let $E$ be an elliptic curve over $k$ with good ordinary reduction at all primes of $k$ above $\ell$. Suppose that $X(E/k_\infty)$ is a torsion $\Lambda_k$-module and that $\mu_{E/k}=0$. 

Let $k'/k$ be an $\ell$-power Galois extension of number fields such that $E$ does not have additive reduction at the primes that ramify in $k'/k$. Then $\mu_{E/k'}=0$ and $X(E/k'_\infty)$ is a torsion $\Lambda_{k'}$-module. Furthermore, 
$$
\lambda_{E/k'}=[k'_\infty:k_\infty]\lambda_{E/k} + \sum_{w\in P_1} (e_{k'_\infty/k_\infty}(w)-1) + 2 \sum_{w\in P_2} (e_{k'_\infty/k_\infty}(w)-1) 
$$
where 
$$
\begin{aligned}
P_1&= \{ w\in M_{k'} : w\nmid \ell \mbox{ and } E \mbox{ has split multiplicative reduction at }w\}  \\
P_2&= \{ w\in M_{k'} : w\nmid \ell \mbox{, } E \mbox{ has good reduction at } w \mbox{, and }E(k'_{\infty,w})[\ell^\infty]\ne (0)\}
\end{aligned}
$$
and $e_{k'_\infty/k_\infty}(w)$ denotes the corresponding ramification index.
\end{theorem}

The following lemma allows one to give a simple alternative description of the set of places $P_2$ of the previous theorem.
\begin{lemma}[Dokchitser-Dokchitser, cf.\ Lemma 3.19 (3) \cite{DD}] \label{LemmaDD} Let $\ell>2$ be a prime and let $k'/k$ be an $\ell$-power Galois extension of number fields. Let $E$ be an elliptic curve over $k$ with good reduction at a prime $v\nmid \ell$ of $k$ and let $w|v$ be a prime of $k'$. Then, $E(k'_{\infty,w})[\ell^\infty]= (0)$ if and only if $\tilde{E}_v(\F_v)[\ell]= (0)$.
\end{lemma}
Theorem \ref{ThmHM} leads to a simple criterion to ensure that, under favorable conditions, the finiteness of the Iwasawa module $X(E/k_\infty)$ is preserved in prime-power degree field extensions (see also Corollary 3.20 in \cite{DD}).
\begin{proposition}[Preserving finiteness of $X(E/k_\infty)$] \label{PropPreserveXfin} Let $\ell>2$ be a prime and let $k$ be a number field. Let $E$ be an elliptic curve over $k$ with good ordinary reduction at all primes of $k$ above $\ell$. Suppose that $X(E/k_\infty)$ is finite.

Let $k'/k$ be a $\ell$-power Galois extension of number fields satisfying the following conditions:
\begin{itemize}
\item[(i)] $E$ has good reduction at each prime $v$ of $k$ that ramifies in $k'/k$, and
\item[(ii)] for each prime $v\nmid \ell$ of $k$ that ramifies in $k'/k$, we have $\tilde{E}_v(\F_v)[\ell]=(0)$.
\end{itemize}
Then $X(E/k'_\infty)$ is finite and $\rk E(k')=0$.
\end{proposition}
\begin{proof} Since $X(E/k_\infty)$ is finite, it is a torsion $\Lambda_k$-module and $\mu_{E/k}=\lambda_{E/k}=0$ (cf.\ Lemma \ref{Lemmalm}). By (i), we can apply Theorem \ref{ThmHM} and we obtain that $X(E/k'_\infty)$ is $\Lambda_{k'}$-torsion and $\mu_{E/k'}=0$. Furthermore, $e_{k'_\infty/k_\infty}(w)=1$ for each $w\in P_1$ by our assumption (i), and $e_{k'_\infty/k_\infty}(w)=1$ for each $w\in P_2$ by (ii) and Lemma \ref{LemmaDD}. Therefore $\lambda_{E/k'}=[k'_\infty : k_\infty]\lambda_{E/k} =0$.

Since $X(E/k'_\infty)$ is $\Lambda_{k'}$-torsion and $\mu_{E/k'}=\lambda_{E/k'}=0$, we obtain that $X(E/k'_\infty)$ is finite (cf.\ Lemma \ref{Lemmalm}). In particular, since $\lambda_{E/k'}=0$ we get $\rk E(k')=0$ (cf.\ Theorem \ref{ThmLambdaRk}).
\end{proof}

In view of Proposition \ref{PropPreserveXfin}, we need a test for finiteness of $X(E/k_\infty)$.


\subsection{Testing finiteness of $X(E/k_\infty)$}

The following test for finiteness of $X(E/k_\infty)$ follows from the theory of the cyclotomic Euler characteristic and it is well-known to experts.
\begin{proposition}[A test for finiteness of $X$] \label{PropTestFinX}Let $\ell$ be a prime and let $k$ be a number field. Let $E$ be an elliptic curve over $k$ with good ordinary reduction at all primes of $k$ above $\ell$. Assume that $E(k)$ is finite and that that $\Sha(E/k)[\ell^\infty]$ is also finite ---the latter happens,  for instance, if $\Sha(E/k)[\ell]=(0)$. Then $X(E/k_\infty)$ is $\Lambda_k$-torsion. Furthermore, consider the quantity
$$
\varpi(E/k):= \frac{\Tam(E/k)\cdot \# \Sha(E/k)[\ell^\infty] }{\#E(k)^2} \cdot \prod_{v|\ell} \#\tilde{E}_v(\F_v)^2.
$$
Then $\varpi(E/k)\in \Z_\ell$, and if $\varpi(E/k)$ is an $\ell$-adic unit, then $X(E/k_\infty)$ is finite.
\end{proposition}
\begin{proof}
Since $E(k)$ and  $\Sha(E/k)[\ell^\infty]$ are finite, the exact sequence \eqref{EqnSha}
$$
0\to E(k)\otimes_\Z\Q_\ell/\Z_\ell\to \Sel_{\ell^\infty}(E/k)\to \Sha(E/k)[\ell^\infty]\to 0
$$
gives that $\Sel_{\ell^\infty}(E/k)$ is also finite ---in fact, we get $\Sel_{\ell^\infty}(E/k)\simeq \Sha(E/k)[\ell^\infty]$. In particular, Theorem \ref{ThmMazurFin} gives that $X(E/k_\infty)$ is $\Lambda_k$-torsion. 

Thus, we can consider the characteristic polynomial $f_{E/k}\in\Z_\ell[T]$ (for a fixed choice of topological generator $\gamma\in\Gamma_k$). By Proposition \ref{Propf0} and the fact that $\Sel_{\ell^\infty}(E/k)$ is finite, we get that $f_{E/k}(0)$ is a non-zero $\ell$-adic integer. Theorem \ref{Thmf0}  gives $\varpi(E/k)\sim_\ell f_{E/k}(0)$, so $\varpi(E/k)\in\Z_\ell$. 

Suppose that $\varpi(E/k)$ is an $\ell$-adic unit.  Let us recall that $f_{E/k}=\ell^{\mu_{E/k}}\prod_j f_j^{m_j}$ for certain monic irreducible polynomials $f_j$ having all their non-leading coefficients in $\ell\Z_\ell$. Since $f_{E/k}(0)\sim_\ell \varpi(E/k)$ is an $\ell$-adic unit, we deduce that in fact $\mu_{E/k}=0$ and that the product $\prod_j f_j^{m_j}$ equals $1$ because it is empty ---for otherwise, the constant term of each $f_j$ would be divisible by $\ell$. Hence $\lambda_{E/k}=0$. Since $\mu_{E/k}=\lambda_{E/k}=0$, we get that $X(E/k_\infty)$ is finite (cf.\ Lemma \ref{Lemmalm}).
\end{proof}



\subsection{Chebotarev sets of primes} A set of prime numbers $\Scal$ is called a \emph{Chebotarev set} if there is a Galois extension $K/\Q$ and a conjugacy-stable set $C\subseteq \mathrm{Gal}(K/\Q)$ such that $S$ agrees with the set $\{p : \Frob_p\in C\}$ up to a finite set. Finite unions, finite intersections, and complements of Chebotarev sets are again Chebotarev. Let us recall that the Chebotarev density theorem states that if $\Scal$ arises from $K$ and $C$ as above, then the limit
$$
\delta(\Scal)=\lim_{x\to \infty}\frac{\#\Scal\cap[1,x]}{\pi(x)}
$$
exists and equals $\#C/[K:\Q]$, where $\pi(x)$ is the number of prime numbers  $p\le x$. The quantity $\delta(\Scal)$ is called the \emph{density} of $\Scal$.

\subsection{Auxiliary primes}  The next result ensures that we have enough primes for the constructions required in the proof of Theorem \ref{ThmRkZero}. As usual, if $E$ is an elliptic curve over $\Q$ and $p$ is a prime of good reduction, we write $a_p(E)=p+1-\# \tilde{E}_p(\F_p)$.
\begin{proposition} \label{PropCheb} Let $\ell$ be a prime number. Let $E$ be an elliptic curve defined over $\Q$ of conductor $N$. Suppose that the residual Galois representation $\rho_{E[\ell]}$ is surjective. Consider the set of prime numbers
$$
\Pcal(E,\ell)=\{p : p\nmid N\mbox{, }p \equiv 1\bmod \ell\mbox{, and }a_p(E)\not\equiv 2\bmod \ell\}.
$$
Then $\Pcal(E,\ell)$ is a Chebotarev set of primes with density 
$$
\delta(\Pcal(E,\ell))=\frac{\ell^2-\ell-1}{(\ell-1)(\ell^2-1)}>0.
$$
Furthermore, for each $p\in\Pcal(E,\ell)$ and each place $v| p$ of $\Q(\mu_\ell)$ we have that $\tilde{E}_v(\F_v)[\ell]=(0)$.
\end{proposition}
\begin{proof} The Weil pairing shows that the $\ell$-division field $K=\Q(E[\ell])$ contains $\Q(\mu_\ell)$. The cyclotomic character $\chi:\Gal(\Q(\mu_\ell)/\Q)\to \F_\ell^\times$ extends to $\chi:\Gal(K/\Q)\to \F_\ell^\times$ by means of the quotient $\Gal(K/\Q)\to \Gal(\Q(\mu_\ell)/\Q)$. We recall that the residual Galois representation $\rho_{E[\ell]}:\Gal(K/\Q)\to GL_2(\F_\ell)$ is injective and for all $p\nmid \ell N$ it satisfies 
\begin{itemize}
\item $\det(\rho_{E[\ell]}(\Frob_p))=\chi(\Frob_p)=p\bmod \ell$, and 
\item $\tr(\rho_{E[\ell]}(\Frob_p))=a_p(E)\bmod \ell$.
\end{itemize}
The map $\rho_{E[\ell]}:\Gal(K/\Q)\to GL_2(\F_\ell)$ is in fact an isomorphism by our surjectivity assumption. Let $C=\{\gamma\in GL_2(\F_\ell) :  \det(\gamma)=1\mbox{ and }\tr(\gamma)\ne 2\}\subseteq GL_2(\F_\ell)$. The set $C$ is conjugacy-stable and it is non-empty (e.g.\ $-I\in C$ when $\ell>2$). In fact, it is an elementary exercise to check that $\#\{\gamma\in SL_2(\F_\ell) : \tr(\gamma)=2\}=\ell^2$, which gives
$$
\# C =\#SL_2(\F_\ell)-\#\{\gamma\in SL_2(\F_\ell) : \tr(\gamma)=2\}= \ell(\ell^2-1)- \ell^2=\ell(\ell^2-\ell-1).
$$
Since we can write
$$
\Pcal(E,\ell)=\{p : p\nmid N \mbox{  and }\rho_{E[\ell]}(\Frob_p)\in C\}
$$
we see that $\Pcal(E,\ell)$ is a Chebotarev set of density
$$
\delta(\Pcal(E,\ell))=\frac{\#C}{\# GL_2(\F_\ell)} = \frac{\ell(\ell^2-\ell-1)}{(\ell^2-1)(\ell^2-\ell)}=\frac{\ell^2-\ell-1}{(\ell-1)(\ell^2-1)}>0.
$$
Finally, let $p\in \Pcal(E,\ell)$. Since $p\equiv 1\bmod \ell$ we have that $p$ splits completely in $\Q(\mu_\ell)$, hence, for each place $v|p$ of $\Q(\mu_\ell)$ we have that the residue field of $\Q(\mu_\ell)$ at $v$ is $\F_v=\F_p$. Therefore $\tilde{E}_v(\F_v)\simeq \tilde{E}_p(\F_p)$ because $E$ has good reduction at $p\nmid N$, and we get
$$
\# \tilde{E}_v(\F_v) = p+1-a_p(E)\equiv 1+1-a_p(E)\not\equiv 0\bmod \ell
$$
because $p\in\Pcal(E,\ell)$. Thus, the $\ell$-torsion of $\tilde{E}_v(\F_v)$ is trivial.
\end{proof}



\subsection{Rank zero over extensions} Let us keep the notation of assumptions of Theorem \ref{ThmRkZero}.
\begin{proof}[Proof of Theorem \ref{ThmRkZero}] The fact that $\Pcal(E,\ell)$ is a Chebotarev set of primes of the asserted density follows from Proposition \ref{PropCheb} and our Condition \eqref{B}. 

Let $k=\Q(\mu_\ell)$. By Condition \eqref{A}, $E$ has good ordinary reduction at each place $v|\ell$ of $k$. Furthermore, $\ell$ ramifies completely in $k/\Q$, thus, there is only one place $v|\ell$ and the residue field satisfies $\F_v=\F_\ell$, which in particular gives $\tilde{E}_v(\F_v)\simeq \tilde{E}_\ell(\F_\ell)$ because $E$ has good reduction at $\ell$.

By Condition \eqref{D} we get $\Sha(E/k)[\ell^\infty]=(0)$, which together with Condition \eqref{E} gives that the integer
$$
\Tam(E/k)\cdot  \# \Sha(E/k)[\ell^\infty]  \cdot \prod_{v|\ell} \#\tilde{E}_v(\F_v)^2
$$
is not divisible by $\ell$. By Conditions \eqref{A}, \eqref{C}, and \eqref{D} we can apply Proposition \ref{PropTestFinX} and we deduce that $X(E/k_{\infty})$ is a finite group because the $\ell$-adic integer $\varpi(E/k)$ is not divisible by $\ell$.

Let $a>1$ be an $\ell$-power free integer supported on $\Pcal(E,\ell)$ and let $k'=\Q(\mu_\ell, \sqrt[\ell]{a})=k(\sqrt[\ell]{a})$. The extension $k'/k$ is Galois of degree $\ell$. Note that $\ell>2$, $E$ has good ordinary reduction at each prime of $k$ above $\ell$ (by Condition \eqref{A}), and $X(E/k_\infty)$ is finite, so that we can apply Proposition \ref{PropPreserveXfin}. For this, note that if a prime $v\in M_k$ ramifies in $k'/k$ then it divides $a$, hence, $v$ divides a prime $p\in\Pcal(E, \ell)$ and we deduce that (cf.\ Proposition \ref{PropCheb}):
\begin{itemize}
\item  $E$ has good reduction at $v$ (as $p\nmid N$ by definition of $\Pcal(E,\ell)$), and
\item $\tilde{E}_v(\F_v)[\ell]=(0)$.
\end{itemize}
Therefore, Proposition \ref{PropPreserveXfin} gives that $X(E/k'_\infty)$ is finite. In particular $\lambda_{E/k'}=0$ (cf. Lemma \ref{Lemmalm}), $\Sel_{\ell^\infty}(E/k')$ is finite, and $\rk E(k')=0$ (cf.\ Theorem \ref{ThmLambdaRk}).
\end{proof}


\section{Achieving positive rank: Heegner points} \label{SecKL}

In this section we present a theorem by Kriz and Li \cite{KrizLi} regarding congruences of Heegner points on quadratic twists of elliptic curves, along with some additional facts tailored for our intended applications. The results of Kriz and Li are particularly convenient for producing many explicit quadratic twists with positive rank for a given elliptic curve over $\Q$. We also work-out a concrete example that will be needed later.

\subsection{Heegner points} \label{SecHeegner} Let $N$ be a positive integer.  Attached to $N$ there is the modular curve $X_0(N)$ over $\Q$ whose non-cuspidal points classify cyclic degree $N$ isogenies of elliptic curves. The cusp $i\infty\in X_0(N)(\Q)$ determines an embedding of the modular curve into its Jacobian $j_N: X_0(N)\to J_0(N)$ defined over $\Q$.

A quadratic imaginary field $K$ is said to satisfy the \emph{Heegner hypothesis} for $N$ if each prime $p|N$ splits in $K$. If $K$ satisfies the Heegner condition for $N$ we can choose a factorization $(N)=\nfrak \nfrak'$ in $O_K$ with $\nfrak'$ the complex conjugate of the ideal $\nfrak$, and one has $O_K/\nfrak\simeq \Z/N\Z$. The map 
$$
\C/O_K\to \C/\nfrak^{-1}
$$
is a cyclic degree $N$ isogeny of complex elliptic curves, and it defines a non-cuspidal point $Q_{N,K, \nfrak}\in X_0(N)(H)$ where $H$ is the Hilbert class field of $K$. Mapping to $J_0(N)$ and taking trace for the extension $H/K$, we obtain the $K$-rational point
$$
P_{N,K,\nfrak}=\sum_{\sigma\in \Gal(H/K)} j_N(Q_{N,K, \nfrak}^\sigma)\in J_0(N)(K).
$$
The point $P_{N,K,\nfrak}$ is the \emph{Heegner point in $J_0(N)$ attached to $K$}. It is independent of the choice of factorization $N=\nfrak\nfrak'$ up to sign and adding torsion, and we will denote it by $P_{N,K}$ because this ambiguity is irrelevant for our discussion.

\subsection{Modular parametrizations} \label{SecModular} Let $E$ be an elliptic curve over $\Q$ of conductor $N$. The modularity theorem \cite{BCDT, TaylorWiles, Wiles} gives a non-constant map $\varphi_E:X_0(N)\to E$ defined over $\Q$, with $\varphi_E(i\infty)=0_E$ where $0_E$ is the neutral point of $E$. Such a map $\varphi_E$ is called modular parametrization. It induces a morphism of abelian varieties $\pi_E : J_0(N)\to E$ defined over $\Q$ satisfying $\varphi_E=\pi_E\circ j_N$. Of course $\varphi_E$ is not unique; however, if the degree of $\varphi_E$ is assumed to be minimal (which we don't) then $\varphi_E$ is determined up to multiplication by $-1$ in $E$.

Let $\omega_E\in H^0(E,\Omega^1_{E/\Q})$ be a global N\'eron differential on $E$; it is unique up to sign. Let $f(q)=q+a_2q^2+...\in S_2(\Gamma_0(N))$ be the unique Fourier-normalized Hecke newform attached to $E$ by the modularity theorem. Then there is a rational number $c(\varphi_E, \omega_E)$ such that
$$
\varphi_E^*\omega_E=c(\varphi_E,\omega_E)\cdot f(q)\frac{dq}{q}.
$$
The absolute value of $c(\varphi_E,\omega_E)$ only depends on $\varphi_E$ and it will be denoted by $c(\varphi_E)$. By considerations on the formal completion of the standard integral model of $X_0(N)$ at $i\infty$, it is easily seen that $c(\varphi_E)$ is an integer. In concrete examples, the Manin constant of $\varphi_E$ can be computed.

A modular parametrization $\varphi_E:X_0(N)\to E$ is \emph{optimal} if the kernel of the induced map $\pi_E:J_0(N)\to E$ is connected. If an optimal modular parametrization exists for $E$ we say that $E$ is optimal (sometimes referred to as strong Weil curve). Every $E$ is isogenous over $\Q$ to an optimal elliptic curve. It is a conjecture of Manin that if $\varphi_E:X_0(N)\to E$ is optimal, then $c(\varphi_E)=1$. Manin's conjecture is proved in a number of cases, e.g.\ when the conductor $N$ is odd and squarefree (this is Corollary 4.2  in \cite{MazurIsog} together with Th\'eor\`eme A in \cite{AbbesUllmo}). See \cite{ARS} and the references therein for more details on Manin's conjecture.

\subsection{Logarithms} \label{SecLog} Let $E$ be an elliptic curve over $\Q$ and let $p$ be a prime. By integration, every differential $\omega\in H^0(E,\Omega^1_{E/\Q})$ uniquely determines a logarithm map
$$
\log_{p,\omega} : E(\Q_p)\to \Q_p
$$
which satisfies the following:
\begin{itemize}
\item For all $a,b\in E(\Q_p)$, $\log_{p,\omega}(a+b)=\log_{p,\omega}(a)+\log_{p,\omega}(b)$.
\item If $\omega\ne 0$ then the kernel of $\log_{p,\omega}$ is $E(\Q_p)_{tor}$.
\end{itemize}
Concretely, the construction is as follows:
Choose a global minimal Weierstrass equation over $\Z$
\begin{equation}\label{EqGenW}
y^2+a_1xy+a_3y=x^3+a_2x^2+a_4x+a_6
\end{equation}
and let $t=-x/y$. Then $t$ is a local parameter at the neutral element  $0_E\in E(\Q)$ and, moreover, a local point $a\in E(\Q_p)$ is congruent to $0_E$ modulo $p$ if and only if $t(a)\in p\Z_p$. Let us write $\omega= (b_0+b_1t+b_2t^2+...)dt$.  Given any $a\in E(\Q_p)$, choose $m$ such that $ma$ is congruent to $0_E$ modulo $p$. Then
$$
\log_{p,\omega}(a)=\frac{1}{m}\sum_{r= 1}^\infty \frac{b_{r-1}\cdot t(ma)^r}{r}
$$
where the series converges because $t(ma)\in p\Z_p$. The value of $\log_{p,\omega}(a)$ is independent of the choice of $m$. We note that the series is the formal integral of $\omega$.

We will consider in the following situation: $K/\Q$ is a quadratic field such that $p$ splits in $K$ and $a\in E(K)$ is a point (in fact, we will take $p=2$). The choice of a prime $\pfrak|p$ in $O_K$ is equivalent to the choice of an embedding $\sigma:K\to \Q_p$. Such a $\sigma$ induces an inclusion $\sigma: E(K)\to E(\Q_p)$ and the quantity $\log_{p,\omega} \sigma(a)$ is thus defined.

For explicit computations with logarithms, it will be convenient to recall the following formulas from Section IV.1 in \cite{SilvermanAdv}. The global N\'eron differential over $\Z$ for $E$ is (in affine coordinates of the model \eqref{EqGenW}) given by
$$
\omega_E=\frac{dx}{2y+a_1x +a_3}
$$
where $a_j\in\Z$ are the coefficients in the global minimal Weierstrass equation \eqref{EqGenW}. This differential expressed in terms of the local parameter $t=-x/y$ at $0_E$ is given by 
\begin{equation}\label{EqDiff}
\omega_E=(1+a_1t+(a_1^2+a_2)t^2 + (a_1^3+2a_1a_2+2a_3)t^3+...)dt\in \Z[[t]]dt
\end{equation}
where the coefficients of the power series are integers that can be computed in terms of the coefficients $a_j$ from \eqref{EqGenW}. For our purposes these first few coefficients suffice.

\subsection{Auxiliary primes}

For an elliptic curve $E/\Q$ of conductor $N$ and a quadratic field $K$, let us define the sets of prime numbers
$$
\begin{aligned}
\Qcal(E,K)&=\{q : q\nmid 2N,\, q\mbox{ splits in }K, \mbox{ and } a_q(E)\equiv 1\bmod 2 \}\\
\Qcal_+(E,K)&=\{q : q\in \Qcal(E,K)\mbox{ and }q\equiv 1\bmod 4\}\\
\Qcal_-(E,K)&=\{q : q\in \Qcal(E,K)\mbox{ and }q\equiv -1\bmod 4\}
\end{aligned}
$$
and observe that $\Qcal(E,K)=\Qcal_+(E,K)\cup \Qcal_-(E,K)$.
\begin{lemma} \label{LemmaPrimes} The sets $\Qcal(E,K)$, $\Qcal_+(E,K)$, and $\Qcal_-(E,K)$ are Chebotarev sets of primes. If $E[2](\Q)\ne (0)$ then the three sets are empty.

Suppose that $E[2](\Q)=(0)$ and $K\ne \Q(\sqrt{-1})$.
\begin{itemize}
\item[(i)]  If $\Gal(\Q(E[2])/\Q)\simeq \Z/3\Z$ then
$$
\delta(\Qcal(E,K))=1/3,\quad \delta(\Qcal_+(E,K))=1/6, \mbox{ and }\, \delta(\Qcal_-(E,K))=1/6.
$$
\item[(ii)] If $\Gal(\Q(E[2])/\Q)\simeq S_3$ and $\Q(E[2])$ does not contain the fields $\Q(\sqrt{-1})$ and $K$, then
$$
\delta(\Qcal(E,K))=1/6,\quad \delta(\Qcal_+(E,K))=1/12, \mbox{ and }\, \delta(\Qcal_-(E,K))=1/12.
$$
\item[(iii)] If $\Gal(\Q(E[2])/\Q)\simeq S_3$ and $\Q(E[2])$ contains the field $\Q(\sqrt{-1})$, then $\Qcal_-(E,K)$ is empty, $\Qcal_+(E,K)=\Qcal(E,K)$, and  $\delta(\Qcal(E,K))=1/6$.
\end{itemize}
\end{lemma}
\begin{proof} The condition $a_q(E)\equiv 1\bmod 2$ at a prime $q\nmid N$ is equivalent to the condition
$$
\tr (\rho_{E[2]}(\Frob_q))=1\in \F_2
$$
for the residual Galois representation $\rho_{E[2]}:\Gal(\Q(E[2])/\Q)\to GL_2(\F_2)\simeq S_3$. The only elements of $GL_2(\F_2)$ with trace equal to $1$ are 
$$
\left[\begin{array}{cc} 1&1\\1&0\end{array}\right]\, \mbox{ and }\, \left[\begin{array}{cc} 0&1\\1&1\end{array}\right]
$$
i.e.\ the two $3$-cycles in $S_3$. 

The three sets of primes are empty if there is a rational $2$-torsion point, since in this case every prime $q\nmid 2N$ satisfies $a_q(E)\equiv \tr(\rho_{E[2]}(\Frob_q))\equiv 0\bmod 2$.

In case (i), since $\Q(E[2])$ does not contain quadratic fields, the result follows from the Chebotarev density theorem and the fact that the congruence conditions modulo $4$ are splitting conditions for the extension $\Q(\sqrt{-1})/\Q$.

In cases (ii) and (iii) we have $\delta(\Qcal(E,K))=1/6$ by the Chebotarev density theorem. Case (ii) follows by linear disjointness of the relevant fields. 

In case (iii) it only remains to show that $\Qcal_-(E,K)$ is empty. Note that $\Q(E[2])$  does not contain $K$ because a hexic extension of $\Q$ contains a unique quadratic subfield. The only non-trivial morphism $s: GL_2(\F_2)\simeq S_3\to \mu_2$ is the sign map, and it sends the trace $1$ matrices (i.e. $3$-cycles of $S_3$) to $1\in \mu_2$. By uniqueness of $s$, the composition $s\circ \rho_{E[2]}:\Gal(\Q(E[2])/\Q)\to \mu_2$ is the same as the map induced by the non-trivial quadratic character $\chi:\Gal(\Q(\sqrt{-1})/\Q)\to\mu_2$ via the restriction map $\Gal(\Q(E[2])/\Q)\to \Gal(\Q(\sqrt{-1})/\Q)$. Since each prime $q\equiv -1\bmod 4$ has $\chi(\Frob_q)=-1$, we get that $\Qcal_-(E,K)$ is empty.
\end{proof}
Case (iii) in Lemma \ref{LemmaPrimes} can in fact happen and we must avoid it for our intended applications (cf. the remarks after Corollary \ref{CoroKL}). For instance, the elliptic curve $E$ of Cremona label $121a1$ has the property that $\Gal(\Q(E[2])/\Q)\simeq S_3$  and $\Q(\sqrt{-1})\subseteq \Q(E[2])$. Here is the data for this $E$ and the first $20$ primes $q\nmid 2N$
$${\tiny
\begin{array}{c|rrrrrrrrrrrrrrrrrrrr}
q \nmid 2\cdot 121  & 3 & 5 & 7 & 13 & 17 & 19 & 23 & 29 & 31 & 37 & 41 & 43 & 47 & 53 & 59 & 61 & 67 & 71 & 73 & 79  \\
q\bmod 4                & -1 & 1 & -1 & 1 & 1   & -1  & -1  & 1   & -1  & 1    & 1  & -1  & -1 &  1  & -1  & 1   & -1  & -1  & 1  & -1  \\
a_q(E)                    & 2  & 1 & 2 & -1 & 5   & -6  &   2  & -9 &  -2  & -3  & 5  & 0   & 2   &  9  & 8   & -6   &  2  & 12  & 2 & 10 
\end{array}}
$$ 
Already for these first few values we see that when $q\equiv -1\bmod 4$ we get $a_q(E)$ even, while for $q\equiv 1\bmod 4$ we can have $a_q(E)$ even or odd depending on whether $\rho_{E[2]}(\Frob_q)\in GL_2(\F_2)\simeq S_3$ is the identity or a $3$-cycle respectively.

Let us point out that in Lemma \ref{LemmaPrimes} the condition that $K$ is not contained in $\Q(E[2])$ will not be a problem for us, due to the following observation.
\begin{lemma} Let $E$ be an elliptic curve over $\Q$ of conductor $N$ satisfying $\Gal(\Q(E[2])/\Q)\simeq S_3$. There is only one quadratic field contained in $\Q(E[2])$ and it does not satisfy the Heegner hypothesis for $2N$.
\end{lemma}
\begin{proof} Since $[\Q(E[2]):\Q]=6$, there is only one quadratic field $F\subseteq \Q(E[2])$. The primes that ramify in $F$ divide $2N$, hence the result.
\end{proof}

\subsection{Congruences of Heegner points, after Kriz and Li}
The following is a re-statement of Theorem 4.3 in \cite{KrizLi}, keeping track of some choices that are made along the proof in \emph{loc.\ cit.}
\begin{theorem}[Kriz-Li \cite{KrizLi}] \label{ThmKL}
Suppose $E/\mathbb{Q}$ is an elliptic curve with $E(\mathbb{Q})[2]=(0)$. Consider a modular parametrization $\varphi_E:X_0(N)\to E$ and the corresponding quotient $\pi_E:J_0(N)\to E$. Let $K$ be an imaginary quadratic field satisfying the Heegner hypothesis for $2N$ and let $\sigma:K\to \mathbb{Q}_2$ be an embedding. Assume that 
\begin{equation}\label{CondKL}
\frac{\#\tilde{E}_2^{ns}(\mathbb{F}_2)}{2\cdot c(\varphi_E)}\cdot \log_{2,\omega_E}(\sigma(\pi_E(P_{N,K})))\in \Z_2^\times.
\end{equation}
Then for each squarefree integer $d\equiv 1\bmod 4$ supported on $\Qcal(E,K)$ the following holds:
\begin{itemize}
\item[(i)]  Both $E$ and $E^{d}$ have rank at most $1$ over $\Q$, and the rank part of the Birch and Swinnerton-Dyer conjecture holds for them.
\item[(ii)] $\rk E(\Q)\ne \rk E^d(\Q)$ if and only if $\psi_d(-N)=-1$, where $\psi_d$ is the quadratic Dirichlet character attached to the quadratic field $\Q(\sqrt{d})$. 
\end{itemize}
\end{theorem}

Let us make some observations regarding condition \eqref{CondKL}.
\begin{itemize}

\item $\log_{2,\omega_E}$ depends on the choice of the N\'eron differential $\omega_E$ only up to sign, which is irrelevant for condition \eqref{CondKL}.

\item The choice of $\varphi_E$ affects both $c(\varphi_E)$ and $\pi_E(P_{N,K})$, but the quantity in \eqref{CondKL} remains the same up to sign.

\item $P_{N,K}\in J_0(N)(K)$ also depends on the choice of a factorization $(N)=\nfrak\nfrak'$ in $O_K$, but only up to multiplication by $-1$ and adding torsion. Hence $\log_{2,\omega_E}(\sigma(\pi_E(P_{N,K})))$ remains the same up to sign.

\item The proof in \cite{KrizLi} also shows that the quantity in \eqref{CondKL} is a $2$-adic integer. So, the condition actually is about indivisibility by $2$.

\end{itemize}

It is worth pointing out that Theorem \ref{ThmKL} (cf.\ Theorem 4.3 in \cite{KrizLi}) is a consequence of a very general congruence formula for Heegner points. In fact, the theory in \cite{KrizLi} (for $p=2$) actually shows that the quantity in \eqref{CondKL} is congruent modulo $2$ to the analogous quantity for the twisted elliptic curve $E^d$. Thus, condition \eqref{CondKL} implies that the analogous indivisibility holds for $E^d$ and, in particular, the corresponding Heegner points in $E(K)$ and $E^d(K)$ are non-torsion (as their logarithms are non-zero). This, together with known results on the Birch and Swinnerton-Dyer conjecture  (due to Kolyvagin \cite{Kol}, Gross-Zagier \cite{GZ}, Murty-Murty \cite{MM}, and Bump-Friedberg-Hoffstein \cite{BFH}) implies item (i) in Theorem \ref{ThmKL}. 

Regarding item (ii) of Theorem \ref{ThmKL}, the condition $\psi_d(-N)=-1$ ensures that the global root numbers of $E$ and $E^d$ are different, so that (ii) follows from (i). The next result is Corollary 5.11 in \cite{KrizLi}, which shows that this parity condition takes a particularly simple form when \eqref{CondKL} holds and the Tamagawa factor $c_2(E/\Q)$ is odd.
\begin{corollary}\label{CoroKL} Let us keep the same notation and assumptions of Theorem \ref{ThmKL} and let $\Delta_E\in\Z$ be the minimal discriminant of $E$. Assume condition \eqref{CondKL} for $E$ and that $c_2(E/\Q)$ is odd.  We have $\rk E(\Q)\ne \rk E^d(\Q)$ if and only if $\Delta_E>0$ and $d<0$.
\end{corollary}
In particular, if $\rk E(\Q)=0$, then Corollary \ref{CoroKL} can only ensure $\rk E^d(\Q)>0$ for negative values of $d$, under favorable conditions. Later we will be interested in taking $d=-q$ for $q$ a prime. So, it will be crucial that the set $\Qcal_-(E,K)$ be non-empty; we must avoid case (iii) in Lemma \ref{LemmaPrimes}.



\section{Ranks and Hilbert's tenth problem} \label{SecProof}

In this section we prove Theorem \ref{ThmMainIntro}. For this, we will need an auxiliary elliptic curve to which we will apply Theorem \ref{ThmRkZero} (with $\ell=3$) and Theorem \ref{ThmKL}. The chosen elliptic curve is 557b1 in Cremona's notation, which has minimal Weierstrass equation over $\Z$
$$
 y^2+y=x^3-x^2-268x+1781.
$$

When using this elliptic curve, we will simply indicate ``$E=557b1$''.


\subsection{Compatibility in the Birch and Swinnerton-Dyer conjecture}\label{SecBSD}

The Birch and Swinnerton-Dyer conjecture for abelian varieties over number fields (as formulated in \cite{Tate}) enjoys a number of compatibility properties, such as compatibility under product, isogenies \cite{Cassels, Tate, MilneADT}, and base change \cite{MilneInv}. The same compatibilities hold regarding the finiteness of the  $p$-primary part of the Shafarevich-Tate group and the $p$-adic valuation of the conjectural special value formula (cf.\ Remark 7.4 in \cite{MilneADT}). We will use these compatibilities to reduce the verification of hypothesis in Theorem \ref{ThmRkZero} for $\ell=3$ to a computation involving elliptic curves over $\Q$. 

This has a practical purpose, since exact numerical computations with elliptic curves over number fields are more difficult than over $\Q$. More concretely, we will need to check that $\Sha(E/\Q(\sqrt{-3})[3]$ is trivial in a specific case, which involves a $3$-descent to compute the relevant $3$-Selmer group. However, the algorithm to compute $3$-Selmer groups seems to be implemented in Magma only for elliptic curves over $\Q$.

We recall that if $E$ is an elliptic curve over $\Q$ with $L$-function $L(E,s)$ and real period $\Omega_E$, then the modularity of $E$ implies that $L(E,s)$ is entire and that $L(E,1)/\Omega_E$ is a rational number that can be exactly, explicitly, and efficiently computed by the theory of modular symbols: it has small denominator (a divisor of $2\cdot c(\varphi_E)\cdot  \# E(\Q)_{tor}$ for any modular parametrization $\varphi_E$), so a good numerical approximation suffices.

\begin{proposition}\label{PropBSD}
Let $E$ be a semi-stable elliptic curve over $\Q$ with good reduction at the prime $3$. Let $E'=E^{-3}$ be the quadratic twist of $E$ by $-3$. Suppose that 
\begin{itemize}
\item[(i)] $\rk E(\Q)=0$ and $\rk E'(\Q)=0$.
\item[(ii)] $\Sha(E/\Q)[3^\infty]$ and $\Sha(E'/\Q)[3^\infty]$ are finite.
\item[(iii)]  The $3$-adic valuation of the rational numbers $L(E,1)/\Omega_E$ and $L(E',1)/\Omega_{E'}$ is as predicted by the $3$-part of the Birch and Swinnerton-Dyer conjecture.
\end{itemize}
Then $\rk E(\Q(\mu_3))=0$, the group $\Sha(E/\Q(\mu_3))[3^\infty]$ is finite, and we have
$$
\frac{L(E,1)}{\Omega_E}\cdot \frac{L(E',1)}{\Omega_{E'}}\sim_3 \frac{\Tam(E/\Q(\mu_3))\cdot \# \Sha(E/\Q(\mu_3))[3^\infty]}{\#E(\Q(\mu_3))_{tor}^2}.
$$
Furthermore, if  $3\nmid \#\tilde{E}_3(\F_3)$, then $3\nmid \#E(\Q(\mu_3))_{tor}$.
\end{proposition}
\begin{proof} Note that $\Q(\mu_3)=\Q(\sqrt{-3})$. By (i) we have $\rk E(\Q(\mu_3))=\rk E(\Q)+\rk E'(\Q)=0$.

Let $A$ be the abelian surface over $\Q$ obtained as Weil restriction of scalars of $E/\Q(\mu_3)$ to $\Q$. Then $A$ is isogenous over $\Q$ to $E\times E'$. By (ii) and compatibilities, we get that $\Sha(E/\Q(\mu_3))[3^\infty]$ is finite.

 The period of $E/\Q(\mu_3)$ can be computed using a global N\'eron differential $\omega_E$ of a minimal model of $E$ over $\Q$ (as $E$ is semi-stable), which up to powers of $2$ gives $|\int_{E(\C)}\omega_E\wedge\overline{\omega_E}|$. Up to powers of $2$, this is equal to $\Omega_E\Omega_{E'}\sqrt{3}$ because $E$ has good reduction at $3$ (cf.\ \cite{Pal}). 
 
 By (iii) and compatibilities, the $3$-part of the Birch and Swinnerton-Dyer conjecture holds for $E/\Q(\mu_3)$. Thus we have
$$
\frac{\Tam(E/\Q(\mu_3))\cdot \# \Sha(E/\Q(\mu_3))[3^\infty]}{\#E(\Q(\mu_3))_{tor}^2}\sim_3 \frac{L(E/\Q(\mu_3),1)\sqrt{3}}{|\int_{E(\C)}\omega_E\wedge\overline{\omega_E}|}
$$
where the $\sqrt{3}$ comes from the discriminant of $\Q(\mu_3)$ (see \cite{DDbsd} for an explicit statement of the conjectural Birch and Swinnerton-Dyer formula for abelian varieties over number fields). Artin formalism gives $L(E/\Q(\mu_3),s) = L(E,s)L(E',s)$, and the claimed formula follows.

Finally, let $v$ be the only place of $\Q(\mu_3)$ over $3$, and note that $v|3$ is ramified.  Since $E$ has good reduction at $3$, we have that $E(\Q(\mu_3))_{tor}$ injects in $\tilde{E}_v(\F_v)\simeq  \tilde{E}_3(\F_3)$, proving the final claim.
\end{proof}

We remark that the semi-stability condition is just for convenience and it can be relaxed with a more careful analysis; for our purposes this is enough.


\subsection{Keeping rank zero in many cubic extensions}

\begin{lemma} \label{LemmaRkZero557b1} Let $E=557b1$. Define the set of primes
$$
\Pcal = \{p :  p\equiv 1\bmod 3 \mbox{ and } a_p(E)\not\equiv 2\bmod 3\}. 
$$
Then $\Pcal$ is a Chebotarev set of primes of density $5/16$ and for each $p\in\Pcal$ we have 
$$
\rk E(\Q(\sqrt[3]{p}))=0.
$$
\end{lemma}
\begin{proof} It suffices to check the conditions \eqref{A} - \eqref{E} in Theorem \ref{ThmRkZero} for $\ell=3$. 

In lmfdb.org \cite{LMFDB} one checks that $E$ has good ordinary reduction at $3$ and that $\rho_{E[3]}:G_\Q\to GL_2(\F_3)$ is surjective, thus Conditions \eqref{A} and \eqref{B} are satisfied.

Consider the quadratic twist $E'=E^{-3}$, which has Cremona label 5013a1. Let us apply Proposition \ref{PropBSD}. The following code in Magma checks that both $\Sel_3(E/\Q)$ and $\Sel_3(E'/\Q)$ are trivial: 
\begin{verbatim}
E1:=EllipticCurve("557b1");
E2:=EllipticCurve("5013a1");
Order(ThreeSelmerGroup(E1));
Order(ThreeSelmerGroup(E2));
      1
      1
\end{verbatim}
In particular, we obtain:
\begin{itemize}
\item $\rk E(\Q)=0$ and $\rk E'(\Q)=0$
\item $\Sha(E/\Q)[3]$ and $\Sha(E/\Q)[3]$ are trivial. Hence, $\Sha(E/\Q)[3^\infty]$ and $\Sha(E/\Q)[3^\infty]$ are also trivial.
\end{itemize}
The exact analytic order of the Shafarevich-Tate groups of $E$ and $E'$ over $\Q$ is $1$ (see for instance lmfdb.org). This is the order predicted by the BSD conjecture, and the computation is exact since the rank is $0$.

Therefore, conditions (i), (ii) and (iii) in Proposition \ref{PropBSD} hold. In addition, $\# \tilde{E}_3(\F_3)= 2$ so $3\nmid \#E(\Q(\mu_3))_{tor}$ and we get
$$
\frac{L(E,1)}{\Omega_E}\cdot \frac{L(E',1)}{\Omega_{E'}}\sim_3 \Tam(E/\Q(\mu_3))\cdot \#\Sha(E/\Q(\mu_3))[3^\infty].
$$
The number on the left is known to be a rational number of small denominator (cf.\ Paragraph \ref{SecBSD}), so the approximation 
$$
\frac{L(E,1)}{\Omega_E}\cdot \frac{L(E',1)}{\Omega_{E'}}\approx \frac{4.14294}{4.14294}\cdot \frac{1.36207 }{0.68103} \approx 2.00001
$$
implies that the number is in fact $2$. Thus, Proposition \ref{PropBSD} and the fact that $\# \tilde{E}_3(\F_3)= 2$ show that conditions \eqref{C}, \eqref{D}, and \eqref{E} of Theorem \ref{ThmRkZero} hold for $E$ and $\ell=3$.
\end{proof}

\subsection{Increasing the rank in many quadratic extensions}

Our next goal is to produce many twists of $E=557b1$ having positive rank over $\Q$. For this we will use the results in Section \ref{SecKL}, which in particular involve the computation of certain logarithm on $p$-adic points of elliptic curves. The next observation will allow us to truncate the relevant power series with controlled $p$-adic precision in our computations.
\begin{lemma}\label{LemmaConv} Let $F(t)=b_0+b_1 t+b_2 t^2+...\in \Z_p[[t]]$ and let $a\in p\Z_p$. Suppose that for certain $n\ge 1$ the integer
$$
m=v_p\left(ab_0 + \frac{a^2b_1}{2} +...+ \frac{a^{n}b_{n-1}}{n}\right)
$$
satisfies $m<n-(\log n)/(\log p)$. Then 
$$
v_p\left(\int_0^a F(t)dt\right) =m.
$$
\end{lemma}
\begin{proof}
Observe that for all $r\ge 1$ we have 
$$
v_p(a^rb_{r-1}/r)\ge r + 0 - v_p(r)\ge r-\frac{\log r}{\log p}
$$ 
because $a\in p\Z_p$ and $b_j\in \Z_p$. Using this for $r>n$ we get
$$
v_p\left(ab_0 + \frac{a^2b_1}{2} +...+ \frac{a^{n}b_{n-1}}{n}-\int_0^a F(t)dt\right) >m,
$$
hence the result.
\end{proof}
Next we produce the required twists of positive rank for $E=557b1$.
\begin{lemma} \label{LemmaRkOne557b1} Let $E=557b1$. Define the set of primes
$$
\Qcal = \{ q :  q\equiv -1\bmod 4,\ (q/7)=1,\ a_q(E)\equiv 1\bmod 2\}. 
$$
Then $\Qcal$ is a Chebotarev set of primes of density $1/12$ and for each $q\in\Qcal$ we have 
$$
\rk E(\Q(\sqrt{-q}))=1.
$$
\end{lemma}
\begin{proof}
We are going to apply Theorem \ref{ThmKL} and Corollary \ref{CoroKL} to $E$ and $K=\Q(\sqrt{-7})$. Then we will use Lemma \ref{LemmaPrimes} to show that the set of primes $\Qcal$ has the required properties.

First, we note that $K$ satisfies the Heegner hypothesis for $2N=2\cdot 557$. This is because a prime $\ell$ splits in $K$ if and only if $\ell$ is a quadratic residue modulo $7$.

The elliptic curve $E$ has good reduction at $2$ and $\#\tilde{E}_2(\F_2)=1$. It admits an optimal modular parametrization $\varphi: X_0(557)\to E$ (see lmfdb.org) and for this $\varphi$ the Manin constant is $c(\varphi)=1$  as $557$ is odd and squarefree, cf.\ Paragraph \ref{SecModular} (alternatively, there are efficient algorithms to directly compute $c(\varphi)$). Let $\pi: J_0(557)\to E$ be the corresponding optimal quotient. Choosing an embedding $\sigma: K\to \Q_2$, we need to show that Condition \eqref{CondKL} holds, i.e.
\begin{equation}\label{CondKL557b1}
v_2\left(\log_{2,\omega_E} \sigma(\pi(P_{N,K}))\right)=1
\end{equation}
where $v_2$ is the $2$-adic valuation in $\Q_2$. According to Table 2 in \cite{KrizLi} this condition indeed holds. However, the details are not included in \emph{loc.\ cit.}\ so we do the computation here. For this we use the explicit description of $\log_{2,\omega_E}$ recalled in Paragraph \ref{SecLog}.

The equation 
\begin{equation}\label{EqW}
 y^2+y=x^3-x^2-268x+1781
\end{equation}
is the reduced global minimal Weierstrass model for $E$ over $\Z$. The global N\'eron differential for this model as a power series on $t=-x/y$ (local parameter at $0_E$) is
$$
\omega_E= (1-t^2+2t^3+...)dt
$$ 
where we used \eqref{EqDiff} for the model \eqref{EqW} to compute the first few coefficients. All the coefficients in this power series are integers.

Consider the $2$-adic neighborhood of $0_E$ given by $U=\{a\in E(\Q_2) : t(a)\in 2\Z_2\}$. The logarithm map $\log_{2,\omega_E}:U\to\Q_2$ is given by
\begin{equation}\label{EqLog557b1}
\log_{2,\omega_E}(a)=t(a) -\frac{1}{3}t(a)^3 + \frac{1}{2}t(a)^4 +...
\end{equation}
The Hilbert class-field of $K=\Q(\sqrt{-7})$ is $K$ itself since the class number is $1$ so, in fact, we have (cf.\ Paragraph \ref{SecHeegner})
$$
\pi(P_{N,K})=\varphi(Q_{N,K,\nfrak})\in E(K).
$$ 
Using Sage, we do the following: For a suitable choice of $\nfrak$, we compute the point $P=\varphi(Q_{N,K,\nfrak})$ in coordinates over $K$ for the model \eqref{EqW}. Then $t=-x(P)/y(P)$ is computed and we choose a valuation $v$ on $K$ extending the $2$-adic valuation of $\Q$ ---this is the same as choosing the embedding $\sigma: K\to \Q_2$. With all of this, we compute $v(t(P))$ and $v(t(P)-t(P)^3/3)$. Here is the code:
\begin{verbatim}
E=EllipticCurve('557b1');
P=E.heegner_point(-7).point_exact(100);
t=-P[0]/P[1];
K=t.parent();
u=QQ.valuation(2);
vK=u.extensions(K); v=vK[0];
v(t); v(t-t^3/3)
      1
      1
\end{verbatim}
Since $v(t(P))=1$ we see that $\pi(P_{N,K})\in U$ and we can use \eqref{EqLog557b1} to compute the logarithm of $P$. Since $v(t(P)-t(P)^3/3)=1$ we can apply Lemma \ref{LemmaConv} with $p=2$, $F(t)=\omega_E$, and $n=3$, obtaining $v(\log_{2,\omega_E}(P))=1$. This proves that \eqref{CondKL557b1} holds.

Since $E(\Q)[2]=(0)$ and \eqref{CondKL557b1} holds, we can apply Theorem \ref{ThmKL}. Furthermore, we have $c_2(E)=1$ because $E$ has good reduction at $2$ ($N$ is odd), so we can apply Corollary \ref{CoroKL}. We have that $\rk E(\Q)=0$ and $\Delta_E=557>0$. Therefore, for all squarefree integers $d\equiv 1 \bmod 4$ supported on $\Qcal(E,K)$ with $d<0$ we have $\rk E^d(\Q)=1$.

We observe that $\Qcal = \Qcal_-(E,K)$ and for each $q\in \Qcal$, the integer $d=-q$ satisfies the previously required conditions, hence $\rk E^{-q}(\Q)=1$. Since $\rk E(\Q)=0$ we deduce $\rk E(\Q(\sqrt{-q})=1$ for each $q\in \Qcal$.

It only remains to check that $\Qcal$ (i.e.\ $\Qcal_-(E,K)$) is a Chebotarev set of density $1/12$ (in particular, that it is non-empty!). We apply Lemma \ref{LemmaPrimes}. The degree of the field $L=\Q(E[2])$ and the discriminant of its only quadratic subfield can be computed on Sage as follows:
\begin{verbatim}
E=EllipticCurve('557b1');
L.<a> = E.division_field(2); 
L.degree();
L.subfields(2)[0][0].discriminant()
      6
      557
\end{verbatim}
Thus, we are in case (ii) of Lemma \ref{LemmaPrimes}, which proves what we wanted.
\end{proof}


\subsection{Proof of the main result} 

\begin{proof}[Proof of Theorem \ref{ThmMainIntro}] Consider the elliptic curve $E=557b1$ and let $\Pcal$ and $\Qcal$ be the sets of primes given by Lemmas \ref{LemmaRkZero557b1} and \ref{LemmaRkOne557b1}. For each $p\in\Pcal$ and each $q\in\Qcal$ we have that $\rk E(\Q(\sqrt[3]{p}))=0$ and $\rk E(\Q(\sqrt{-q}))=1$. By Proposition \ref{PropMain} with $M=\Q$, $F=\Q(\sqrt[3]{p})$ and $L=\Q(\sqrt{-q})$ we get that the extension $\Q(\sqrt[3]{p}, \sqrt{-q})/\Q(\sqrt[3]{p})$ is integrally Diophantine.

Since $\Q(\sqrt[3]{p})$ has exactly one complex place, the extension $\Q(\sqrt[3]{p})/\Q$ is integrally Diophantine by the Pheidas-Shlapentokh-Videla theorem \cite{Pheidas, ShlH10, Videla}. Therefore, Lemma \ref{LemmaTransfer} gives that $\Z$ is Diophantine in $\Q(\sqrt[3]{p},\sqrt{q})$.
\end{proof}


\section{Acknowledgments}

The first author was supported by the FONDECYT Iniciaci\'on en Investigaci\'on grant 11170192, and the CONICYT PAI grant 79170039. The second author was supported by FONDECYT Regular grant 1190442.

We gratefully acknowledge the hospitality of the IMJ-PRG in Paris during our visit in February 2019, where the final ideas of the project were developed. We heartily thank Lo\"ic Merel for this invitation and for several enlightening discussions on elliptic curves. 

The final technical details of this work were worked out and presented in May 2019 at the AIM meeting ``Definability and Decidability problems in Number Theory''. We thank the AIM for their hospitality and the participants for their valuable feedback. Specially, we thank Karl Rubin and Alexandra Shlapentokh for their interest and technical remarks.

Comments by Matias Alvarado, Chao Li, and Barry Mazur on an earlier version of this manuscript are gratefully acknowledged.


\end{document}